\documentclass[11pt,a4paper]{article}
\usepackage{amsmath,amsthm,amssymb}

\newtheorem{theorem}{Theorem}
\newtheorem{proposition}[theorem]{Proposition}
\newtheorem{lemma}[theorem]{Lemma}

\newtheorem{Conjecture}[theorem]{Conjecture}

\title{{\bf Ideal class groups of certain real quadratic number fields 
and factorizations of values of certain quadratic polynomials}}
\author{
St\'ephane R. LOUBOUTIN\\
Aix-Marseille Universit\'e, CNRS, Centrale Marseille, I2M,\\ 
Marseille, France\\
stephane.louboutin@univ-amu.fr}
\date{\today}
\begin{document}
\bibliographystyle{alpha}
\maketitle
\footnotetext{

Class group. 
Class number. 
Quadratic polynomial.}

\begin{abstract}
We determine the odd primes $p$ for which the number of distinct prime divisors of $d-x^2$ is at most $2$ 
for every positive odd integers $x\leq\sqrt{p}$. 
We also determine the odd primes $p$ for which the number of distinct prime divisors of $2p-x^2$ is at most $2$ 
for every positive even integers $x\leq\sqrt{2p}$. 
To this end, we show that these properties of $p$ and $2p$ impose restrictions 
on the class numbers and fundamental units of the real quadratic fields ${\mathbb Q}(\sqrt{p})$ 
and ${\mathbb Q}(\sqrt{2p})$.
\end{abstract}

\section{Introduction}
We keep the notation introduced in \cite{LouCMJ76}.
Let $\omega(n)$ denote the number of distinct prime divisors of an integer $n\geq 2$. 
For $d\geq 2$ we set
$$M_{odd}'(d)
:=\max\{\omega(d-x^2);\text{ $1\leq x\leq\sqrt d$ and $x$ odd}\}$$ 
and 
$$M_{even}'(d)
:=\max\{\omega(d-x^2);\text{ $2\leq x\leq\sqrt d$ and $x$ even}\}.$$ 
In \cite{GicaarXiv}, A. Gica determines all odd primes $p$ for which $M_{odd}'(p)\leq 2$, 
with one possible exception when $p\equiv 3\pmod 8$. 
In \cite[Theorem 1.9]{LouCMJ76} we proved that if $d\equiv 2\pmod 4$ is a positive square-free integer 
such that $M_{even}'(d)\leq 2$, 
then the class number of the real quadratic number field ${\mathbb K}={\mathbb Q}(\sqrt{d})$ divides $16$.
Unlike the results established in Propositions 
\ref{casedp=5mod8},
\ref{casedp=7mod8},
\ref{casedp=3mod8},
\ref{cased2p=1mod4},
\ref{cased2p=7mod8},
and \ref{cased2p=3mod8}, 
we were unable to prove that $M_{even}'(d)\leq 2$ implies that the fundamental unit of ${\mathbb K}$ 
is small.
It follows that we could not apply the Brauer-Siegel theorem to deduce that there are only finitely many such integers $d$. 
However, numerical computations readily suggested the following conjecture:

\begin{Conjecture}\label{conjecture}
Let $d\equiv 2\pmod 4$ be a positive square-free integer. 
Then $M_{even}'(d)\leq 2$ if and only if 
$d\in\{2,$ $6,$ $10,$ $14,$ $22,$ $26$, $30,$ $38,$ $42$, $62,$ $110,$ $122,$ $182,$ $278$, $362$, $398\}$ 
($16$ values). 
\end{Conjecture}

Our first aim is to present Gica's determination of the odd primes $p$ for which $M_{odd}'(p)\leq 2$ 
clearly and concisely, 
see Section \ref{Sectiond=p}. 
For example, Section \ref{p=1mod8} is a clearer and shorter presentation of Gica's proof 
for the case where $p\equiv 1\pmod 8$. 
We emphasize, however, that all the ideas in Section \ref{p=1mod8} are due to Gica.
Our second aim, in Section \ref{Sectiond=p}, is to solve Conjecture \ref{conjecture} 
for the case where $d=2p$ is twice a prime, 
with one possible exception of the form $2p=9m^2\pm 8m+2$, in the case where $p\equiv 3\pmod 8$. 
The possible exception in Propositions \ref{casedp=3mod8} and \ref{cased2p=3mod8}
comes from the use of the Siegel-Tatuzawa theorem,.
This potential exception could likely be ruled out using the methods developed in 
\cite{BiroActa106}, \cite{BiroActa107}, \cite{BL} and \cite{LeeActa140}.

\section{Notations}
Let $d>1$ be a positive integer that is not a square. 
We denote by ${\mathbb Z}_{\mathbb K}$ the ring of algebraic integers, 
$d_{\mathbb K}>0$ the discriminant, 
$h_{\mathbb K}$ the class number 
and ${\mathcal Cl}_{\mathbb K}$ the ideal class group 
of the real quadratic number field 
${\mathbb K} 
={\mathbb Q}(\sqrt{d})
=\{x+y\sqrt{d};\ x,y\in {\mathbb Q}\}$. 
Every ideal class in ${\mathcal Cl}_{\mathbb K}$ contains an ideal whose norm is at most the Minkowski bound 
$$M_{\mathbb K}
=\frac{1}{2}\sqrt{d_{\mathbb K}}.$$ 
If ${\mathcal I}$ is a non-zero ideal of ${\mathbb Z}_{\mathbb K}$, 
we let $[{\mathcal I}]\in {\mathcal Cl}_{\mathbb K}$ denote its ideal class. 
If $\ell\geq 2$ is prime, we let ${\mathcal P}_l$ denote either of the prime ideals of ${\mathbb K}$ above $\ell$.
We will use the following lemma:

\begin{lemma}\label{lsnotinert}
Let $d>0$ be square-free. 
If an odd prime $\ell\geq 3$ is non-inert in ${\mathbb K}={\mathbb Q}(\sqrt{d})$, 
i.e., if $\left (\frac{d}{\ell}\right )\neq -1$ (the Legendre symbol), 
then there exists an integer $x\in\{0,1,\cdots,\ell\}$ such that $\ell\mid d-x^2$. 
We may also require that $x$ be even or odd.
\end{lemma}

\begin{proof}
Since $\left (\frac{d}{\ell}\right )$ is equal to $0$ or $+1$, 
$d$ is a square modulo $\ell$, 
i.e., there exists $x\in {\mathbb Z}$ such that $\ell$ divides $d-x^2$. 
By reducing $x$ modulo $\ell$, we may assume that $0\leq x\leq\ell$.
Changing $x$ to $\ell-x$ if necessary, we can force $x$ to be odd or to be even. 
\end{proof}

In the proof of \cite[Lemma 5.3]{LouCMJ76} we omitted the proof that we may assume that $1\leq x,y\leq\sqrt{p}$. 
We therefore provide a more complete statement and proof of the lemma here. 

\begin{lemma}\label{existence}
Let $p\equiv\pm 1\pmod 8$ be a prime. 
Take $d=p$ or $d=2p$.
There exist integers $x,y$ such that $d+x^2=2y^2$ and $1\leq x,y\leq\sqrt{d}$. 
\end{lemma}

\begin{proof}
The ring of algebraic integers of the real quadratic number field ${\mathbb Q}(\sqrt{2})$ 
is ${\mathbb Z}[\sqrt{2}]$. 
For $\alpha =x+y\sqrt{2}\in {\mathbb Z}[\sqrt{2}]$, its conjugate is $\alpha'=x-y\sqrt{2}$ 
and its norm is $N(\alpha) =\alpha\alpha'=x^2-2y^2$.
The ring ${\mathbb Z}[\sqrt{2}]$ is a principal ideal domain 
whose fundamental unit $\eta =1+\sqrt{2}$ has norm $-1$. 
Since $p\equiv\pm 1\pmod 8$, the prime $p$ splits in ${\mathbb Z}[\sqrt{2}]$
as a product of two prime principal ideals of norm $2$, say $(p)={\mathcal PP'}$ 
with ${\mathcal P} =(\alpha)$, ${\mathcal P'} =(\alpha')$, 
$p=N({\mathcal P}) =\vert N(\alpha)\vert$ 
and $p=N({\mathcal P'}) =\vert N(\alpha')\vert$. 
By replacing $\alpha$ by $\alpha\eta$ if necessary, 
we may assume that $N(\alpha)=-p$. 
We then define $\beta =(2+\sqrt{2})\alpha$. 
We have $N(\beta)=-2p$.
Therefore, writing $\alpha =x+y\sqrt{2}$ and $\beta =x'+y'\sqrt{2}$, 
we have $p+x^2=2y^2$ and $2p+x'^2=2y'^2$. 

Now, by replacing $\alpha$ by $-\alpha$ if necessary, 
we may assume that $\alpha>0$.
Since ${\mathbb R}_+^* =\cup_{n\in {\mathbb Z}} [\eta^{2n-1},\eta^{2n+1})$, there exists $n\in {\mathbb Z}$ 
such that $\alpha/\sqrt{p}\in [\eta^{2n-1},\eta^{2n+1})$. 
By replacing $\alpha$ by $\alpha\eta^{-2n}$, 
and noticing that $\alpha\eta^{-2n}$ has still norm $-p$,
we may assume that $\alpha/\sqrt{p}\in [\eta^{-1},\eta)$. 
Then $x =(\alpha+\alpha')/2 =(\alpha -p/\alpha)/2$. 
Since $t\mapsto t-p/t$ is increasing on ${\mathbb R}_+^*$, 
we obtain 
$$-\sqrt{p}
=\frac{\sqrt{p}}{2}\left (\frac{1}{\eta}-\eta\right )
\leq x
\leq\frac{\sqrt{p}}{2}\left (\eta-\frac{1}{\eta}\right )
\leq\sqrt{p}$$ 
and $2y^2=p+x^2\leq 2p$. 
Hence, there exist integers $x,y$ such that $p+x^2=2y^2$ and $1\leq x^2,y^2\leq p$. 
By changing the signs of $x$ and $y$ if necessary, we obtain the desired result. 
In the same way, we may assume that $\beta/\sqrt{2p}\in [\eta^{-1},\eta)$. 
Then $x' =(\beta+\beta')/2 =(\beta -2p/\beta)/2$. 
We obtain 
$$-\sqrt{2p}
=\frac{\sqrt{2p}}{2}\left (\frac{1}{\eta}-\eta\right )
\leq x'
\leq\frac{\sqrt{2p}}{2}\left (\eta-\frac{1}{\eta}\right )
\leq\sqrt{2p}$$
and $2y'^2=2p+x'^2\leq 4p$. 
We conclude as before.
\end{proof}

\section{The case where $d$ is prime}\label{Sectiond=p}
\begin{lemma}\label{pl}
Let $p\not\equiv 1\pmod 8$ be an odd prime. 
Suppose that $M_{odd}'(p)\leq 2$. 
Let $\ell\leq M_{\mathbb K}$ be an odd prime that is non-inert in ${\mathbb K}={\mathbb Q}(\sqrt{p})$. 
Then 
$(i)$ $p\not\equiv 7\pmod 8$, 
$(ii)$ $p=x^2+2\ell^2$ for some odd integer $x\in\{1,\cdots,\ell-1\}$ if $p\equiv 3\pmod 8$, 
and $(iii)$ $p=1+4\ell^2$ if $p\equiv 5\pmod 8$.
\end{lemma}

\begin{proof}
By Lemma \ref{lsnotinert}, 
there exists $x\in\{1,\cdots,\ell\}$ odd such that $\ell$ divides $p-x^2$. 
It follows that $p-x^2=2^a\ell^b$, with $b\geq 1$, $a=1$ if $p\equiv 3, 7\pmod 8$, and $a=2$ if $p\equiv 5\pmod 8$. 
Since $\ell\leq M_{\mathbb K}\leq\sqrt{p}<p$, we have $\ell\neq p$. 
Hence, $\ell\nmid x$, 
i.e., $x\in\{1,\cdots,\ell-2\}$.
Now, $p-(2\ell-x)^2=2^a\ell N$, 
with $N=\ell^{b-1}-2^{2-a}(\ell-x)$ odd. 
Moreover, $N\geq l^{b-1}-2(\ell-2)$ $\geq \ell^2-2(\ell -2) =(\ell-2)^2>1$ 
and $\ell\nmid N$ if $b\geq 3$. 
We conclude that $M_{odd}'(p)\geq\omega(p-(2x-\ell)^2)$ $=\omega(2^a\ell N)\geq 3$ if $b\geq 3$. 
Therefore, $b=1$ or $b=2$. 
For $p\equiv 3, 7\pmod{8}$ we have $b=2$ and $p=x^2+2\ell^2$, 
as $b=1$ would give the following contradiction:
$\ell^2\leq M_{\mathbb K}^2 =p=x^2+2\ell\leq (\ell-2)^2+2\ell<\ell^2$. 
It follows that $p=x^2+2\ell^2\equiv 3\pmod 8$. 
Thus, points $(i)$ and $(ii)$ are proved. 
Finally, suppose that $p\equiv 5\pmod 8$. 
Then $\ell\leq M_{\mathbb K} =\frac{1}{2}\sqrt{p}$ implies $\ell\leq 2\ell -x\leq 2\ell\leq\sqrt{p}$. 
Here again, $N$ is odd, and $\omega (2^a\ell N)=\omega(p-(2\ell -x)^2)\leq M_{odd}'(p)\leq 2$ implies $N=\ell^{b-1}+x-\ell =1$. 
As $b=1$ would imply $x=\ell$, we have $b\geq 2$. 
As $\ell^2+x-\ell\geq\ell^2+2-\ell >1$, we have $b<3$. 
We conclude that $b=2$. Then, $N=1$ gives $x=1$. 
We obtain $p=x^2+4\ell^b =1+4\ell^2$, and point $(iii)$ is proved.
\end{proof}

\subsection{The case where $d$ is a prime congruent to $7\pmod 8$}
\noindent\frame{\vbox{
\begin{proposition}\label{casedp=7mod8}
Let $p\equiv 7\pmod 8$ be a prime. 
Set ${\mathbb K}={\mathbb Q}(\sqrt{p})$, 
a real quadratic number field of discriminant $d_{\mathbb K} =4p$.
If $M_{odd}'(p)\leq 2$, 
then $(i)$ $p=n^2-2$ for some odd $n\geq 3$, 
$(ii)$ the odd primes below the Minkowski bound $M_{\mathbb K} =\sqrt{p}$ remain inert in ${\mathbb K}$, 
$(iii)$ $h_{\mathbb K} =1$ 
and $(iv)$ there are only finitely many such primes $p$. 
Moreover, $M_{odd}'(p)\leq 2$ if and only if $p\in\{7, 23, 47, 167\}$.
\end{proposition}
}}

\begin{proof}
By Lemma \ref{existence}, there exist $a,b\in [0,\sqrt{p})$ such that $p+a^2=2b^2$. 
Clearly, $a$ is odd, $b$ is even and $\gcd(a,b)=1$. 
Since $M_{odd}'(p)\leq 2$ and $0<p-a^2 =2(b-a)(b+a)$ with $\gcd(b-a,b+a)=1$, 
it follows that $b-a=1$ and $p=2b^2-a^2=2(a+1)^2-a^2 =n^2-2$ with $n=a+2$. 
Since $n+\sqrt{p}\in {\mathbb Z}_{\mathbb K}$ has norm $2$, 
the prime ramified ideal of ${\mathbb K}$ above $2$ is principal.
Moreover, the odd primes $\ell\leq M_{\mathbb K} =\sqrt{p}$ are inert in ${\mathbb K}$, 
by Lemma \ref{pl}. 
We conclude that $h_{\mathbb K}=1$. 
Since $p+1+\sqrt{p(p+2)}>1$ is a unit in ${\mathbb Z}_{\mathbb K}$ for $p=n^2-2$, 
the regulator of ${\mathbb K}$ is $O(\log d_{\mathbb K})$. 
Hence, there are only finitely many primes $p$ of the form $p=n^2-2$ for which $h_{\mathbb K}=1$, 
by the Brauer-Siegel theorem (e.g., see \cite{Tatuzawa} or \cite{LouCollMath108}). 
Finally, by \cite{LeeActa140}, if $d=n^2-2$ is prime, then $h_{\mathbb K}=1$ if and only if $p\in\{7, 23, 47, 167\}$.
\end{proof}

\subsection{The case where $d$ is a prime congruent to $3\pmod 8$}
\begin{lemma}\label{uniqueness}
A prime $p\equiv 1\text{ or } 3\pmod{8}$ can be written uniquely in the form $p=m^2+2n^2$ for some $m,n\geq 1$. 
Note that $m$ and $n$ are odd and coprime.
\end{lemma}

\begin{proof}
The ring ${\mathbb Z}[\sqrt{-2}]$ is euclidean for the norm 
and $-1$ and $1$ are its only invertible elements.
Since $\left (\frac{-2}{p}\right )=+1$ (the Legendre symbol), 
there exists $x\in {\mathbb Z}$ such that $p$ divides $x^2+2=(x+\sqrt{-2})(x-\sqrt{-2})$. 
By Euclid's lemma, we deduce that $p$ is not irreducible in ${\mathbb Z}[\sqrt{-2}]$. 
It follows that $p=\alpha\beta$ in ${\mathbb Z}[\sqrt{-2}]$, 
where $\alpha$ and $\beta$ are non-units and therefore have norms greater than $1$. 
It follows that $p=N(\alpha)=N(\beta)$. 
Writing $\alpha =m+n\sqrt{-2}$ we obtain $p=m^2+2n^2$. 
Conversely, if $p=m^2+2n^2$, 
then $p=(m+n\sqrt{-2})(m+n\sqrt{-2})$ is a decomposition into a product of irreducible elements in 
${\mathbb Z}[\sqrt{-2}]$. 
Since this decomposition is unique 
up to permutation and multiplication of its factors by invertible elements $-1$ and $+1$, 
we obtain the desired result.
\end{proof}

\noindent\frame{\vbox{
\begin{proposition}\label{casedp=3mod8}
Let $p\equiv 3\pmod 8$ be a prime. 
Set ${\mathbb K}={\mathbb Q}(\sqrt{p})$, 
a real quadratic number field of discriminant $d_{\mathbb K} =4p$.
If $M_{odd}'(p)\leq 2$, then either
$(i)$ $p=m^2+2$, 
the odd primes below the Minkowski bound $M_{\mathbb K} =\sqrt{p}$ 
remain inert in ${\mathbb K}$ and $h_{\mathbb K} =1$, 
or $(ii)$ $p=9m^2\pm 8m+2$ for some odd $m\geq 1$, 
the odd primes $\ell\neq 2m\pm 1$ below the Minkowski bound $M_{\mathbb K} =\sqrt{p}$ 
remain inert in ${\mathbb K}$ and $h_{\mathbb K} =1$. 
There are only finitely many such primes $p$. 
Moreover, with one possible exception, of the form $p=9m^2\pm 8m+2$, 
$M_{odd}'(p)\leq 2$ if and only if $p\in\{3, 11, 19, 59, 83, 107, 227\}$.
\end{proposition}
}}

\begin{proof}
Assume that $M_{odd}'(p)\leq 2$. 
Recall that $p$ can be written uniquely in the form $p=m^2+2n^2$ for some odd, coprime integers $m,n\geq 1$, 
by Lemma \ref{uniqueness}.

First, suppose that $n=1$. 
Suppose that there is an odd prime $\ell\leq M_{\mathbb K} =\sqrt{p}$ that is non-inert in ${\mathbb K}$. 
Then $p =x^2 +2\ell^2$ for some odd $x\in {\mathbb Z}$, by Lemma \ref{pl}, 
and $p=m^2+2$. 
This contradicts Lemma \ref{uniqueness}. 
Since $m+\sqrt{p}$ has norm $-2$, 
it follows that the prime ramified ideal ${\mathcal P}_2$ of ${\mathbb K}$ above $2$ is principal 
and that $h_{\mathbb K}=1$.

Second, suppose that $n>1$. 
If $2m-n\geq 1$, then $p-(2n-m)^2=2n(2m-n)$, with $\gcd(n,2m-n)=1$. 
It follows that $2m-n=1$ and $p=m^2+2(2m-1)^2=8m^2-8m+2$. 
Otherwise, $n-2m\geq 1$ and $p-(3m)^2 =2(n-2m)(n+2m)$, 
with $\gcd(n-2m,n+2m)=1$. 
It follows that $n-2m=1$ and $p=m^2+2(2m+1)^2=8m^2+8m+2$. 
Since $9m\pm 4+3\sqrt{9m^2\pm 8m+2}$ has norm $-2$, 
it follows that the prime ramified ideal ${\mathcal P}_2$ of ${\mathbb K}$ above $2$ is principal.
Let $\ell\leq M_{\mathbb K} =\sqrt{p}$ be an odd prime that is non-inert in ${\mathbb K}$. 
There exists $x\in {\mathbb Z}$ with $p =x^2 +2\ell^2$, by Lemma \ref{pl}. 
Then $p=m^2+2(2m\pm 1)^2$ and $p=x^2+2\ell^2$ imply $\ell=2m\pm 1$, by Lemma \ref{uniqueness}. 
Since $3m\pm 1+\sqrt{9m^2\pm 8m+2}$ has norm $\mp (2m\pm 1)=\mp\ell$, 
any prime ideal ${\mathcal P}_\ell$ of ${\mathbb K}$ above the prime $\ell$ is principal. 
Hence $h_{\mathbb K}=1$.

Since $p-1+\sqrt{p(p-2)}$ for $p=n^2+2$ 
and $9p-1+3\sqrt{p(9p-2)}$ for $p=9m^2\pm 8 m+2$ 
are units in ${\mathbb Z}_{\mathbb K}$, 
the regulator of ${\mathbb K}$ is $O(\log d_{\mathbb K})$. 
We conclude that there are only finitely many such primes $p$ of the form $p=n^2+2$ or $p=9m^2\pm 8 m+2$ 
for which $h_{\mathbb K}=1$, 
by the Brauer-Siegel theorem (e.g., see \cite{Tatuzawa} or \cite{LouCollMath108}). 
Finally, by \cite{LeeActa140}, if $d=m^2+2$ is an odd square-free integer 
and ${\mathbb K} ={\mathbb Q}(\sqrt{d})$, 
then $h_{\mathbb K}=2$ if and only if $d\in\{3, 11, 83, 227\}$.
\end{proof}

\subsection{The case where $d$ is a prime congruent to $5\pmod 8$}
\noindent\frame{\vbox{
\begin{proposition}\label{casedp=5mod8}
Let $p\equiv 5\pmod 8$ be a prime. 
Set ${\mathbb K}={\mathbb Q}(\sqrt{p})$, 
a real quadratic number field of discriminant $d_{\mathbb K} =p$.
If $M_{odd}'(p)\leq 2$, then 
$(i)$ $p=4n^2+1$ for some odd $n\geq 1$ 
and the odd primes $\ell\neq n$ below the Minkowski bound $M_{\mathbb K} =\frac{1}{2}\sqrt{p}$ 
remain inert in ${\mathbb K}$, 
or 
$(ii)$ $p=n^2+4$ for some odd $n\geq 1$ 
and the odd primes below the Minkowski bound $M_{\mathbb K} =\frac{1}{2}\sqrt{p}$ remain inert in ${\mathbb K}$.
Therefore, $h_{\mathbb K} =1$ 
and there are only finitely many such primes $p$. 
Moreover, $M_{odd}'(p)\leq 2$ if and only if $p\in\{5, 13 ,29, 37, 53, 101, 173, 197, 293, 677\}$.
\end{proposition}
}}

\begin{proof}
Since $p\equiv 5\pmod 8$, $p$ is a sum of two positive squares, 
i.e., $p=m^2+4n^2$ with $m,n\geq 1$ odd and coprime. 
Thus, $1\leq m,2n\leq\sqrt{p}$ and $\vert m-2n\vert\leq\max(m,2n)\leq\sqrt{p}$. 
Since $M_{odd}'(p)\leq 2$ and $p-(m-2n)^2 =4mn$, 
it follows that $m=1$ or $n=1$, which proves the first assertion. 
Suppose that there is an odd prime $\ell\leq M_{\mathbb K} =\frac{1}{2}\sqrt{p}$ that is non-inert in ${\mathbb K}$. 
Then $p=1+4\ell^2$, by Lemma \ref{pl}, 
and $(2\ell +1+\sqrt{p})/2\in {\mathbb Z}_{\mathbb K}$ has norm $\ell$. 
Therefore, the prime ideals of ${\mathbb K}$ above $\ell$ are principal.
Since the prime $2$ is inert in ${\mathbb K}$, we obtain $h_{\mathbb K}=1$. 
Finally, since $\sqrt{p-1}+\sqrt{p}>1$ for $p=4n^2+1$ 
and $(\sqrt{p-4}+\sqrt{p})/2>1$ for $p=n^2+4$ 
are units in ${\mathbb Z}_{\mathbb K}$, 
the regulator of ${\mathbb K}$ is $O(\log d_{\mathbb K})$. 
Hence, there are only finitely many primes $p$ of the form $p=4n^2+1$ or $p=n^2+4$ for which $h_{\mathbb K}=1$, 
by the Brauer-Siegel theorem (e.g., see \cite{Tatuzawa} or \cite{LouCollMath108}). 
Finally, for $d=4n^2+1$ square-free we have $h_{\mathbb K}=1$ if and only if 
$d\in\{5, 17, 37, 101, 197, 677\}$, by \cite{BiroActa107}, 
and for $d=n^2+4$ square-free we have $h_{\mathbb K}=1$ if and only if 
$d\in\{5, 37, 29, 53, 173, 293\}$, by \cite{BiroActa106}.
\end{proof}

\subsection{The case where $d$ is a prime congruent to $1\pmod 8$}\label{p=1mod8}
\begin{lemma}\label{lsmallisinert}
Let $p\equiv 1\pmod 8$ be a prime. 
Assume that $M_{odd}'(p)\leq 2$. 
Then any odd prime $\ell\leq\sqrt{p/16}$ is inert in ${\mathbb K}={\mathbb Q}(\sqrt{p})$, i.e., 
$\left (\frac{p}{\ell}\right )=-1$ (the Legendre symbol).
\end{lemma}

\begin{proof}
Suppose, for a contradiction, 
that there is an odd prime $\ell\leq\sqrt{p/16}$ that is non-inert in ${\mathbb K}$. 
Then $\ell\neq p$ and there exists an odd positive integer $x$ with $1\leq x\leq\ell -2$ such that $\ell\mid p-x^2$. 
It follows that $p-x^2=2^a\ell^b$ for some integers $a\geq 3$ and $b\geq 1$. 
Assume that $b\geq 2$. 
Since $0<2\ell-x\leq 2\ell +x\leq 3\ell<\sqrt{p}$ 
and $p-(2\ell\pm x)^2\equiv\mp 4\ell x\not\equiv 0\pmod{\ell^2}$, 
we conclude that
 $p-(2\ell -x)^2=2^c\ell$ and $p-(2\ell +x)^2=2^d\ell$ for some $c,d\geq 3$. 
Moreover, $(2\ell+x)^2-(2\ell -x)^2=4\ell x\not\equiv 0\pmod {16}$. 
It follows that either $c$ or $d$ is equal to $3$. 
Since $0<2\ell-x\leq2\ell+x<3\ell$, 
we obtain the following contradiction: 
$16\ell^2<p\leq 9\ell^2+8\ell <16\ell^2$. 
Therefore, $b=1$ and $p-x^2=2^a\ell$. 
In the same way, $0<4\ell -x<4\ell\leq\sqrt{p}$ and $p-(4\ell-x)^2\equiv 8\ell x\not\equiv 0\pmod{\ell^2}$ 
give $p-(4\ell-x)^2=2^e\ell$ for some $e\geq 3$. 
Since $(4\ell-x)^2-x^2\equiv -8\ell x\not\equiv 0\pmod{16}$, 
$a$ or $e$ is equal to $3$. 
If $a=3$ we obtain the following contradiction: 
$16\ell^2<p=x^2+8\ell\leq (\ell-2)^2+8\ell =\ell^2+4\ell +4<16\ell^2$. 
Finally, $d=3$ would give 
$16\ell^2<p=(4\ell -x)^2+8\ell\leq (4\ell-1)^2+8\ell =16\ell^2+1$. 
This would imply $p=16\ell^2+1 =x^2+2^a\ell$, then $x^2\equiv 1\pmod\ell$, 
hence $x=1$ or $x=\ell -1$. 
Since $x$ is odd we would have $x=1$ and the contradiction $16\ell^2+1=2^a\ell +1$.
\end{proof}

\begin{lemma}\label{step1}
Let $p\equiv 1\pmod 8$ be a prime. 
Assume that $M_{odd}'(p)\leq 2$. 
Then either $p=97$, 
or there exist $a\geq 3$ and $b\geq 0$ such that 
$p=(2^{b+1}+2^{a-2}-1)^2+2^a$, 
or $p=(2^{b+1}-2^{a-2}+1)^2+2^a$, 
or $p=(2^{b+1}-2^{a-2}-1)^2-2^a$. 
\end{lemma}

\begin{proof}
There exists $k\geq 1$ such that $(2k-1)^2<p<(2k+1)^2$. There are two cases.\\
{\bf Case 1.} Suppose first that $(2k-1)^2<p<(2k)^2$. 
Then $p-(2k-1)^2=2^a\ell^b$ for some $a\geq 3$, $b\geq 0$ and $\ell\geq 3$ prime. 
Moreover, $2^a\ell^b =p-(2k-1)^2\leq (2k)^2-1-(2k-1)^2$ $=2(2k-1)$. 
Now, $b\geq 1$ and Lemma \ref{lsmallisinert} would give $\ell>\sqrt{p/16}$ and the following contradiction:
$4p<64\ell^2\leq 2^{2a}\ell^{2b}\leq 4(2k-1)^2<4p$. 
We conclude that $b=0$, and 
$$\hbox{$p =(2k-1)^2+2^a$ for some $a\geq 3$.}$$ 
Then, 
$$p-(2^{a-2}+1)^2 
=4AB,$$ 
where $A=k-2^{a-3}>0$, as $2(2k-1)\geq 2^a\ell^b\geq 2^a$, 
and $B=k+2^{a-3}-1>A>0$. 
Hence, $\omega(4AB)\leq M_{odd}'(p)\leq 2$. 
Notice that $A$ and $B$ have opposite parity.\\ 
First, suppose that $\gcd(A,B)$ is not a power of $2$. 
Then, some odd prime $\ell$ divides $A$ and $B$. 
Then $\ell^2>p/16$, by Lemma \ref{lsmallisinert}. 
Clearly, if
$(i)$ $B>A>0$ are of opposite parity and divisible by some odd prime $\ell\geq\sqrt{p/16}$, 
and $(ii)$ $AB<p/4$, 
then $A=\ell$ and $B=2\ell$.
It follows that $\ell =B-A=2^{a-2}-1$. 
Since $p=(2k-1)^2+2^a$ and $p =(2^{a-2}+1)^2+4AB =(2^{a-2}+1)^2+8\ell^2$, 
we deduce that $a$ is even, by Lemma \ref{uniqueness}. 
Since $\ell =2^{a-2}-1$ is prime, $a=4$, $\ell =3$, and $p=(2^{a-2}+1)^2+8\ell^2 =97$.\\
Second, suppose that $\gcd(A,B)$ is a power of $2$. 
Since $\omega(4AB)\leq 2$, 
we conclude that $A$ or $B$ is a power of $2$. 
If $A=2^b$, we get $k=A+2^{a-3} =2^b+2^{a-3}$ and $$p=(2k-1)^2+2^a =(2^{b+1}+2^{a-2}-1)^2+2^a.$$
If $B=2^b$, we get $k=B-2^{a-3}+1 =2^b-2^{a-3}+1$ and $$p=(2k-1)^2+2^a=(2^{b+1}-2^{a-2}+1)^2+2^a.$$

{\bf Case 2.} Now, assume that $(2k)^2<p<(2k+1)^2$. 
Suppose an odd prime $\ell$ divides $(2k+1)^2-p$. 
Then $\ell\geq\sqrt{p/16}$ by Lemma \ref{lsmallisinert}, 
and $8\ell$ divides $(2k+1)^2-p$. 
We obtain the following contradiction: 
and $2\sqrt{p}>4k=(2k+1)^2-(2k)^2-1\geq (2k+1)^2-p\geq 8\ell$. 
Hence, $(2k+1)^2-p$ is a power of $2$ and 
$$\hbox{$p =(2k+1)^2-2^a$ for some $a\geq 3$.}$$ 
Then,
$$p-(2^{a-2}-1)^2 
=4AB,$$ 
where $A=k-2^{a-3}>0$, 
and $B=k+2^{a-3}+1>A>0$. 
As in Case 1 we deduce that either $p=73$, 
or $\gcd (A,B)$ is a power of $2$. 
Assume that $p\neq 73$.
Using $\omega(4AB)\leq 2$, we deduce that $A$ or $B$ is a power of $2$.
If $A=2^b$, we get $k=A+2^{a-3} =2^b+2^{a-3}$ and 
$$p=(2k+1)^2-2^a 
=(2^{b+1}+2^{a-2}+1)^2-2^a
=(2^{a-2}+2^{b+1}+1)^2+2^{b+3}.$$
\end{proof}

The following lemmas are related to the notion of quadratic residue covers of families of real quadratic fields, 
as developed in \cite[Section 4]{LMWCJM45}.

\begin{lemma}\label{step2}
$(i)$. Let $d\geq 2$ be of the form 
$d=d_+(a,b):=(2^{b+1}+2^{a-2}-1)^2+2^a$ 
or $d=d_-(a,b):=(2^{b+1}-2^{a-2}+1)^2+2^a$ 
for some integers $b\geq 0$ and $a\geq 3$. 
Then $\left (\frac{d}{\ell}\right )\neq -1$ for some prime $\ell\in E=\{3, 5, 7, 13, 17, 109, 241, 433\}$.\\
$(ii)$. Let $d\geq 2$ be of the form 
$d=d(a,b):=(2^{b+1}-2^{a-2}-1)^2-2^a$ 
for some integers $b\geq 0$ and $a\geq 3$. 
Then $\left (\frac{d}{\ell}\right )\neq -1$ for some prime $\ell\in E =\{3, 5, 7, 13, 17, 19, 37, 73, 109, 241\}$. 
\end{lemma}

\begin{proof}
Since $2^{72}\equiv 1\pmod\ell$ for $\ell\in E$, 
the function $(a,b)\mapsto d=d_+(a,b)$ is $72$-periodic in $a$ and $b$. 
A straightforward computation using Maple shows that, 
for every pair $(a,b)$ satisfying $3\leq a\leq 74$ and $0\leq b\leq 71$, 
we have $\left (\frac{d_+(a,b)}{\ell}\right )\neq -1$ for some $\ell\in E$. 
The same argument holds for $d_-(a,b)$ and $d(a,b)$.
The desired results follow.
\end{proof}

\noindent\frame{\vbox{
\begin{proposition}\label{casedp=1mod8}
Let $p\equiv 1\pmod 8$ be a prime. 
Set ${\mathbb K}={\mathbb Q}(\sqrt{p})$, 
a real quadratic number field of discriminant $d_{\mathbb K} =p$.
If $M_{odd}'(p)\leq 2$, 
then ${\mathcal Cl}_{\mathbb K}$ is a cyclic group 
generated by the ideal class of either of the two prime ideals of ${\mathbb K}$ above the prime $2$. 
Moreover, we have $M_{odd}'(p)\leq 2$ if and only if 
$p\in\{17, 41, 73, 89, 97, 113, 137, 233, 257, 353, 593, 857,$ $1097, 1217, 1553, 2273, 5297\}$.
\end{proposition}
}}

\begin{proof}
Assume that $M_{odd}'(p)\leq 2$.
Let $\ell\leq M_{\mathbb K} =\frac{1}{2}\sqrt{p}$ be an odd prime that is non-inert in ${\mathbb K}$. 
Hence $\ell\neq p$ and $\ell$ splits in ${\mathbb K}$.
By Lemma \ref{lsnotinert}, there exists an odd integer $x\geq 1$, with $1\leq x\leq l$, such that $\ell\mid p-x^2$. 
Since $p-x^2\equiv 0\pmod 8$ we have $p-x^2=2^a\ell^b$ for some $a\geq 3$ and $b\geq 1$. 
Assume that $b\geq 2$. 
Then $\ell\leq 2\ell -x\leq 2\ell\leq\sqrt{p}$ . 
Since $p-(2l-x)^2 =4\ell (2^{a-2}\ell^{b-1}+x-\ell)$, 
it follows that $2^{a-2}\ell^{b-1}+x-\ell$ which is even and not divisible by $\ell$ must be a perfect power of $2$, 
i.e., $p-(2\ell -x)^2 =2^c\ell$ for some $c\geq 3$. 
We conclude that either $p-x^2=2^a\ell$ or $p-(2\ell-x)^2=2^c\ell$, for some $a,c\geq 3$. 
Since $(x+\sqrt{p})/2\in {\mathbb Z}_{\mathbb K}$ and $(2\ell-x+\sqrt{p})/2\in {\mathbb Z}_{\mathbb K}$ 
have norms $2^{a-2}\ell$ and $2^{c-2}\ell$, 
the ideal classes of the prime ideals above $\ell$ belong to the cyclic group 
generated by the ideal class of either of the two prime ideals of ${\mathbb K}$ above the prime $2$. 
The first assertion is proved. 
For the second assertion, if $M_{odd}'(p)\leq 2$ then 
either $p=97$ or there exists an odd prime $\ell\leq 433$ that is non-inert in ${\mathbb K}$, by Lemmas \ref{step1} and \ref{step2}. 
By Lemma \ref{lsmallisinert} we obtain $p\leq 16\cdot 433^2<3\cdot 10^6$. 
An direct computation yields the desired result.
\end{proof}

\section{The case where $d$ is twice a prime}\label{Sectiond=2p}
\subsection{The case where $d=2p$ with $p$ a prime congruent to $1\pmod 4$}
\begin{lemma}\label{dl}
Let $p\geq 3$ be an odd prime. 
Set $d=2p$ and ${\mathbb K} ={\mathbb Q}(\sqrt{d})$.
Suppose that $M_{even}'(d)\leq 2$. 
Let $\ell\leq\sqrt{d}$ be an odd prime that is non-inert in ${\mathbb K}$. 
Then either $d=\ell^2+1$ and $p\equiv 1\pmod 4$, 
or $d=2\ell^2+x^2$ for some even integer $x\in [2,\ell-1]$ and $p\equiv 1\text{ or } 3\pmod8$.
\end{lemma}

\begin{proof}
There exists $x$ even in $[0,l]$ such that $\ell\mid 2p-x^2$, 
by Lemma \ref{lsnotinert}. 
We have $\ell\neq p$, as $\ell\leq\sqrt{2p}$. 
Hence $2\leq x\leq\ell -1$. 
Since $M_{even}'(d)\leq 2$, we have $d-x^2=2\ell^a$ for some $a\geq 1$. 
Suppose that $a\geq 3$ and let us derive a contradiction. 
We have $2\ell-x\geq 2\ell-(\ell-1)=\ell-1\geq 2$ 
and $(2\ell-x)^2\leq (2\ell)^2\leq 2\ell^3\leq 2\ell^a=d-x^2\leq d$. 
It follows that $M_{even}'(d)\geq\omega (d-(2\ell-x)^2) =\omega(2\ell(\ell^{a-1}+2x-2l))\geq 3$, 
as $\ell\nmid x$.
Hence, either $a=1$ or $a=2$. 
If $a=1$ then 
$\ell^2\leq d=2\ell +x^2\leq 2\ell+(\ell-1)^2 =\ell^2+1$ 
and $d=\ell^2+1$, 
which implies $p=(\ell^2+1)/2\equiv 1\pmod 4$. 
If $a=2$, then $d=2\ell^2+x^2$ and $p=d/2=\ell^2+2(x/2)^2\equiv 1\text{ or } 3\pmod 8$.
\end{proof}

\noindent\frame{\vbox{
\begin{proposition}\label{cased2p=1mod4}
Let $p\equiv 1\pmod 4$ be a prime. 
Set ${\mathbb K}={\mathbb Q}(\sqrt{2p})$, 
a real quadratic number field of discriminant $d_{\mathbb K} =8p$.
If $M_{even}'(2p)\leq 2$, 
then $(i)$ $p\equiv 5\pmod 8$, 
$(ii)$ $2p=m^2+1$ for some odd $m\geq 3$, 
$(iii)$ the odd primes $\ell\neq m$ below the Minkowski bound $M_{\mathbb K} =\sqrt{2p}$ 
remain inert in ${\mathbb K}$, 
$(iv)$ $h_{\mathbb K}=2$, 
and $(v)$ there are only finitely many such primes $p$. 
Moreover, we have $M_{even}'(2p)\leq 2$ if and only if $2p\in\{10, 26, 122, 362\}$.
\end{proposition}
}}

\begin{proof}
Let $p\equiv 1\pmod 4$ be a prime such that $M_{even}'(2p)\leq 2$.

First, assume that there exists a prime $p\equiv 1\pmod 8$ such that $M_{even}'(2p)\leq 2$. 
We derive a contradiction. 
There exist positive integers $x,y\leq\sqrt{p}$ such that $p+x^2=2y^2$, 
by Lemma \ref{existence}. 
Notice that $x$ and $y$ are odd and coprime. 
We have $x<4y/3$, 
because $x\geq\ 4y/3$ would give $2y^2=p+x^2\geq p+16y^2/9$, 
and the contradiction $y^2>9p/2>p$. 
Moreover, we have $x>1$. 
Indeed, $x=1$ would give $p+1=2y^2$ and $2p=m^2+1$, 
which would imply $(m-2y)(m+2y)=m^2-4y^2=-3$. 
Therefore, we would have $m-2y=-1$ and $m+2y=3$, i.e., $m =y =1$ 
and the following contradiction: $2p =m^2+1=2$. 
Now, $2p-(2y-2x)^2 =4y^2-2x^2-(2y-2x)^2 =2x(4y-3x)$, 
where $x>1$ is odd and $4y-3x>0$ is odd and coprime to $x$. 
Since $M_{even}'(2p)\leq 2$, it follows that $4y-3x=1$. 
Hence $p+1=2y^2-x^2+1=2y'^2$, where $y'=(x+3)/4$. 
It follows that $y'$ is a positive integer 
and, as above, $p+1=2y'^2$ and $2p=m^2+1$ give $p=2$, a contradiction.

Second, $p=a^2+b^2$ is a sum of two squares. 
Therefore, $2p=(a+b)^2+(a-b)^2 =m^2+n^2$ is also a sum of two squares. 
Clearly, $m$ and $n$ are odd and coprime. 
We may assume that $m>n\geq 1$. 
Since $\omega(2p-(m-n)^2)=\omega(2mn)\leq M_{even}'(2p)\leq 2$, 
we conclude that either $m=1$ or $n=1$. 
Since $m>n\geq 1$, we have $n=1$ and $2p=m^2+1.$ 

Third, assume that $p\equiv 5\pmod 8$. 
Let $\ell\leq M_{\mathbb K}=\sqrt{2p}$ be an odd prime that is non-inert in ${\mathbb K}$. 
Then $2p=\ell^2+1$, by Lemma \ref{dl}. 
It follows that $\ell =m$.

Fourth, since $\ell+\sqrt{2p}$ has norm $-2\ell$, 
the product ${\mathcal P}_2{\mathcal P}_l$ is principal. 
It follows that $[{\mathcal P}_l]=[{\mathcal P}_2]^{-1} =[{\mathcal P}_2]$ 
for any odd prime $\ell\leq M_{\mathbb K}$ that is non-inert in ${\mathbb K}$. 
Hence $h_{\mathbb K}=1$ or $2$, 
depending on whether ${\mathcal P}_2$ is principal. 
Finally, since ${\mathcal P}_2$ is not principal, we conclude that $h_{\mathbb K}=2$. 
Indeed, $a^2-2pb^2=\pm 2$ would give $a=2a'$, $2a'^2-pb^2=\pm 1$, 
$b$ would be odd and we would obtain the following contradiction: 
$2a'^2=pb^2\pm 1\equiv 4\text{ or } 6\pmod 8$.

Fifth, since $\sqrt{2p-1}+\sqrt{2p}>1$ 
is a unit in ${\mathbb Z}_{\mathbb K}$ for $p=m^2+1$, 
the regulator of ${\mathbb K}$ is $O(\log d_{\mathbb K})$. 
We conclude that there are only finitely many such primes $p$ for which $h_{\mathbb K}=2$, 
by the Brauer-Siegel theorem (e.g., see \cite{Tatuzawa} or \cite{LouCollMath108}). 

Finally, by \cite{BL}, if $d=m^2+1$ is an even square-free integer and ${\mathbb K} ={\mathbb Q}(\sqrt{d})$, 
then $h_{\mathbb K}=2$ if and only if $d\in\{10, 26, 122, 362\}$.
\end{proof}

\subsection{The case where $d=2p$ with $p$ a prime congruent to $3\pmod 4$}
\noindent\frame{\vbox{
\begin{proposition}\label{cased2p=3mod8}
Let $p\equiv 3\pmod 8$ be a prime. 
Set ${\mathbb K}={\mathbb Q}(\sqrt{2p})$, 
a real quadratic number field of discriminant $d_{\mathbb K} =8p$. 
If $M_{even}'(2p)\leq 2$, then either
$(i)$ $2p=m^2+2$, 
the odd primes below the Minkowski bound $M_{\mathbb K} =\sqrt{2p}$ 
remain inert in ${\mathbb K}$ and $h_{\mathbb K} =1$, 
or $(ii)$ $2p=9m^2\pm 8m+2$ for some even $m\geq 2$ 
the odd primes $\ell\neq 2m\pm 1$ below the Minkowski bound $M_{\mathbb K} =\sqrt{2p}$ 
remain inert in ${\mathbb K}$ and $h_{\mathbb K} =1$. 
There are only finitely many such primes $p$. 
Moreover, with one possible exception, 
of the form $d=9m^2\pm 8m+2$,
$M_{even}'(2p)\leq 2$ if and only if $2p\in\{6, 22, 38, 278\}$.
\end{proposition}
}}

\begin{proof}
Assume that $M_{even}'(2p)\leq 2$. 
Set $d=2p$.
As in Lemma \ref{uniqueness}, 
we have that $d$ can be written uniquely in the form $d=m^2+2n^2$ for some coprime integers $m,n\geq 1$, 
with $m$ even and $n$ odd.

First, suppose that $n=1$. 
Suppose that there is an odd prime $\ell\leq M_{\mathbb K} =\sqrt{d}$ that is non-inert in ${\mathbb K}$. 
Since $d =x^2 +2\ell^2$ for some odd $x\in {\mathbb Z}$, by Lemma \ref{dl}, 
and $d=m^2+2$, 
we obtain a contradiction. 
Since $m+\sqrt{d}$ has norm $-2$, 
it follows that the prime ramified ideal ${\mathcal P}_2$ of ${\mathbb K}$ above $2$ is principal 
and that $h_{\mathbb K}=1$.

Second, suppose that $n>1$. 
If $2m-n\geq 1$, then $d-(2n-m)^2=2n(2m-n)$, with $\gcd(n,2m-n)=1$. 
Hence, $2m-n=1$ and $d=m^2+2(2m-1)^2=8m^2-8m+2$. 
Otherwise, $n-2m\geq 1$ and $d-(3m)^2 =2(n-2m)(n+2m)$, 
with $\gcd(n-2m,n+2m)$ $=1$. 
Hence, $n-2m=1$ and $d=m^2+2(2m+1)^2=8m^2+8m+2$. 
Since $9m\pm 4+3\sqrt{9m^2\pm 8m+2}$ has norm $-2$, 
it follows that the prime ramified ideal ${\mathcal P}_2$ of ${\mathbb K}$ above $2$ is principal.
Let $\ell\leq M_{\mathbb K} =\sqrt{d}$ be an odd prime that is non-inert in ${\mathbb K}$. 
There exists $x\in {\mathbb Z}$ with $d =x^2 +2\ell^2$, by Lemma \ref{dl}. 
Then $p=m^2+2(2m\pm 1)^2$ and $p=x^2+2\ell^2$ imply $\ell=2m\pm 1$. 
Since $3m\pm 1+\sqrt{9m^2\pm 8m+2}$ has norm $\mp (2m\pm 1)=\mp\ell$, 
any prime ideal ${\mathcal P}_\ell$ of ${\mathbb K}$ above the prime $\ell$ is principal. 
We conclude that $h_{\mathbb K}=1$.

Since $d-1+\sqrt{d(d-2)}>1$ for $d=n^2+2$ 
and $9p-1+3\sqrt{p(9p-2)}>1$ for $d=9m^2\pm 8 m+2$ 
are units in ${\mathbb Z}_{\mathbb K}$, 
the regulator of ${\mathbb K}$ is $O(\log d_{\mathbb K})$. 
We conclude that there are only finitely many such primes $p$ of the form $p=n^2+2$ or $p=9m^2\pm 8 m+2$ 
for which $h_{\mathbb K}=1$, 
by the Brauer-Siegel theorem (e.g., see \cite{Tatuzawa} or \cite{LouCollMath108}).

Finally, by \cite{LeeActa140}, if $d=m^2+2$ is an even square-free integer 
and ${\mathbb K} ={\mathbb Q}(\sqrt{d})$, 
then $h_{\mathbb K}=2$ if and only if $d\in\{6, 38\}$.
\end{proof}

\noindent\frame{\vbox{
\begin{proposition}\label{cased2p=7mod8}
Let $p\equiv 7\pmod 8$ be a prime. 
Set ${\mathbb K}={\mathbb Q}(\sqrt{2p})$, 
a real quadratic number field of discriminant $d_{\mathbb K} =8p$. 
If $M_{even}'(2p)\leq 2$, then $2p=m^2-2$ for some $m\geq 4$ 
and the odd primes belown the Minkowski bound $M_{\mathbb K} =\sqrt{2p}$ remain inert in ${\mathbb K}$. 
Therefore, $h_{\mathbb K} =1$ 
and there are only finitely many such primes $p$.
Moreover, $M_{even}'(2p)\leq 2$ if and only if $2p\in\{14, 62, 398\}$.
\end{proposition}
}}

\begin{proof}
Set $d=2p$.
By Lemma \ref{existence}, there exist $a,b\in [0,\sqrt{d})$ such that $d+a^2=2b^2$. 
Clearly, $a$ is even and $a\geq 2$. 
Since $M_{even}'(d)\leq 2$ and $d-a^2 =2(b-a)(b+a)$, 
it follows that $b-a=1$ and $2p=2b^2-a^2=2(a+1)^2-a^2 =m^2-2$, 
with $m=a+2$. 
Since $m+\sqrt{p}\in {\mathbb Z}_{\mathbb K}$ has norm $2$, 
the prime ramified ideal of ${\mathbb K}$ above $2$ is principal. 
Since the odd primes $\ell\leq M_{\mathbb K} =\sqrt{d}$ are inert in ${\mathbb K}$, 
by Lemma \ref{dl}, 
we deduce that $h_{\mathbb K}=1$. 
Since $d+1+\sqrt{d(d+2)}>1$ is a unit in ${\mathbb Z}_{\mathbb K}$ for $d=m^2-2$, 
the regulator of ${\mathbb K}$ is $O(\log d_{\mathbb K})$. 
Hence, there are only finitely many square-free integers $d$ of the form $d=m^2-2$ for which $h_{\mathbb K}=1$, 
by the Brauer-Siegel theorem (e.g., see \cite{Tatuzawa} or \cite{LouCollMath108}).
Finally, by \cite{LeeActa140}, if $d=m^2-2$ is an even positive square-free integer 
and ${\mathbb K} ={\mathbb Q}(\sqrt{d})$, 
then $h_{\mathbb K}=2$ if and only if $d\in\{14, 62, 398\}$. 
\end{proof}


{\small
\begin{thebibliography}{?????}
\bibitem[Biro1]{BiroActa106}
A. Bir\'o.
\newblock Yokoi's conjecture.
\newblock {\em Acta Arith.} {\bf 106} (2003), 85--104.

\bibitem[Biro2]{BiroActa107}
A. Bir\'o.
\newblock Chowla's conjecture.
\newblock {\em Acta Arith.} {\bf 107} (2003), 179--194.

\bibitem[BL]{BL}
D. Byeon and J. Lee. 
\newblock Class number 2 problem for certain real quadratic fields of Richaud-Degert type. 
\newblock {\em J. Number Theory} {\bf 128} (2008), 865--883. 
Zbl 1167.11036, MR2400045 (22009a:11222).

\bibitem[Gi1]{GicaarXiv}
A. Gica. 
\newblock An additive problem in connection with some quadratic real fields with class number one. 
https://arxiv.org/html/2609.16960v1

\bibitem[Lee]{LeeActa140}
J. Lee.
\newblock The complete determination of wide Richaud-Degert types which are not 5 modulo 8 with class number one. 
\newblock {\em Acta Arith.} {\bf 140} (2009), 1--29. 
Zbl 1569.11162, MR4818274.

\bibitem[LMW]{LMWCJM45}
S. Louboutin, R.A. Mollin and H.C. Williams.
\newblock Class numbers of real quadratic fields, 
continued fractions, 
reduced ideals, 
prime producing quadratic polynomials 
and quadratic residue covers.
\newblock {\em Canad. J. Math.} {\bf 44} (1992), 824--842.

\bibitem[Lou1]{LouCollMath108}
S. Louboutin.
\newblock Simple proofs of the Siegel-Tatuzawa and Brauer-Siegel theorems.
\newblock {\em Colloq. Math.} {\bf 108} (2007), 277--283. 
Zbl 1114.11090, MR2291638 (2008c:11155).

\bibitem[Lou2]{LouCMJ76}
S. Louboutin.
\newblock Ideal class groups of some quadratic number fields 
and factorization of values of some quadratic polynomials.
\newblock {\em Czechoslovak Math. J.} {\bf 76} (2026), 677--697.

\bibitem[MW]{MW}
R. A. Mollin and H. C. Williams. 
\newblock On a Determination of Real Quadratic Fields of Class Number One 
and Related Continued Fraction Period Length Less than 25.
\newblock {\em Proc. Japan Acad.} {\bf 57}, Ser. A (1991), 20--25.

\bibitem[Tat]{Tatuzawa}
T. Tatuzawa. 
\newblock On a theorem of Siegel. 
\newblock {\em Japan J. Math.} {\bf 21} (1951), 163--178. 
Zbl 0054.02302, MR0051262 (14,452c).
\end{thebibliography}
}
\end{document}